\documentclass{article}
\usepackage{amsmath}
\usepackage{amsthm}
\newtheorem{thm}{Theorem}[section]
\newtheorem{cor}[thm]{Corollary}
\newtheorem{prop}[thm]{Proposition}
\newtheorem{lem}[thm]{Lemma}

\def\zt{\zeta}

\def\si{\sigma}

\def\Ga{\Gamma}
\def\Si{\Sigma}
\def\C{\mathcal C}
\def\I{\mathcal I}
\def\Li{\operatorname{Li}}
\def\Re{\operatorname{Re}}
\begin{document}
\title{Sums of Products of Riemann Zeta Tails}
\author{Michael E. Hoffman\\
\small U. S. Naval Academy\\[-0.8ex]
\small Annapolis, MD 21402 USA\\[-0.8ex]
\small \texttt{meh@usna.edu}}
\date{February 10, 2016
}
\maketitle
\begin{abstract}
A recent paper of O. Furdui and C. V\u alean proves some results
about sums of products of ``tails'' of the series for the Riemann
zeta function.  We show how such results can be proved with weaker
hypotheses using multiple zeta values, and also show how they can be 
generalized to products of three or more such tails.
\end{abstract}
\section{Introduction}
In a recent paper, O. Furdui and C. V\u alean \cite{FV} prove some results
on sums of products of the ``tails'' of the Riemann zeta function, i.e.,
\[
\zt(p)-\sum_{i=1}^n \frac1{i^p} =\sum_{i=n+1}^\infty \frac1{i^p} .
\]
In fact, as we show in \S3 below, there is a simple general result 
(Theorem \ref{sym} below) for the sum
\begin{equation}
\label{spq}
\sum_{n=1}^\infty\left(\zt(p)-\sum_{i=1}^n\frac1{i^p}\right)
\left(\zt(q)-\sum_{j=1}^n\frac1{j^q}\right)
\end{equation}
in terms of the multiple zeta values $\zt(i_1,\dots,i_k)$ defined by
\begin{equation}
\label{mzv}
\zt(i_1,\dots,i_k)=\sum_{n_1>n_2>\dots>n_k\ge 1}\frac1{n_1^{i_1}\cdots n_k^{i_k}}
\end{equation}
(where for convergence $i_1+\dots+i_j>j$ for $j=1,\dots,k$).
The only hypothesis on our result is that $p$ and $q$ are real numbers
exceeding 1 whose sum exceeds 3.
We also show how an integral formula given in \cite{FV} for the sum 
(\ref{spq}) in the case $p=q+1$ can be generalized.
Further, our Theorem \ref{sym} can be generalized to products of three or more
tails, as we show in \S4.
We begin in \S2 with a brief review of multiple zeta values.
\section{Multiple zeta values}
The multiple zeta value (\ref{mzv}) is said to have depth $k$
and weight $i_1+\dots+i_k$. 
Multiple zeta values of general depth were introduced independently
by the author \cite{H} and D. Zagier \cite{Z}, but for depth two they 
were already studied by Euler \cite{E}.
Two results due to Euler are the ``sum theorem''
\begin{equation}
\label{2sumt}
\sum_{i=1}^{n-2}\zt(n-i,i)=\zt(n)
\end{equation}
for integers $n\ge 3$, and the reduction of double zeta values
of the form $\zt(n,1)$, $n\ge 2$ an integer:
\begin{equation}
\label{n1red}
\zt(n,1)=\frac{n}2\zt(n+1)-\frac12\sum_{j=2}^{n-1}\zt(j)\zt(n+1-j) ,
\end{equation}
where an empty sum is understood to be nil.
The following result on depth two multiple zeta values includes
the relation (\ref{2sumt}) as a special case.
If $p,q$ are positive integers with $n=p+q\ge 3$, then
\begin{equation}
\label{binom}
\zt(n)=\sum_{i=p+1}^{n-1}\binom{i-1}{p-1}\zt(i,n-i)+
\sum_{i=q+1}^{n-1}\binom{i-1}{q-1}\zt(i,n-i) .
\end{equation}
Finally, we note that
\begin{equation}
\label{prod}
\zt(n)\zt(m)=\zt(n,m)+\zt(m,n)+\zt(m+n)
\end{equation}
by multiplying series.
\par
The generalization of the sum theorem to arbitrary depth, i.e.,
\begin{equation}
\label{sumt}
\sum_{i_1\ge 2,\ i_j\ge 1,\ i_1+\dots+i_k=n}\zt(i_1,i_2,\dots,i_k)=\zt(n)
\end{equation}
for $n>k\ge 2$, was conjectured by C. Moen (see \cite{H}) and
proven independently by A. Granville \cite{G} and D. Zagier.
Another result for general depth, the ``duality'' property,
is a bit more involved to state.  Let $\Si$ be the function that
assigns to a sequence of positive integers $(i_1,\dots,i_k)$
its sequence of partial sums.  If $\I_n$ is the set of strictly
increasing sequences of positive integers whose last term is
at most $n$, define functions $R_n$ and $C_n$ on $\I_n$ by
\begin{align*}
R_n(a_1,\dots,a_k)&=(n+1-a_k,n+1-a_{k-1},\dots,n+1-a_1)\\
C_n(a_1,\dots,a_k)&=\text{complement of $\{a_1,\dots,a_k\}$
in $\{1,\dots,n\}$, arranged} \\
&\text{in increasing order.}
\end{align*}
Then $\tau=\Si^{-1}R_nC_n\Si$ is a function that takes sequences
$(i_1,\dots, i_k)$ with $i_1>1$ and $i_1+\dots+i_k=n$ to
sequences of the same type, and the duality property is
\begin{equation}
\label{dual}
\zt(\tau(I))=\zt(I) .
\end{equation}
For example, $\zt(2,1,2)=\zt(2,3)$.  This result follows from
the representation of multiple zeta values as iterated integrals
(for which see \cite{Z}).
\par
The following remarks apply only to multiple zeta values 
$\zt(i_1,\dots,i_k)$ whose arguments $i_j$ are all positive integers.  
Every known relation among multiple zeta values with rational
coefficients is homogeneous by weight.
It is known that every multiple zeta value of weight at most 7 can 
be expressed as a polynomial in the ordinary zeta values $\zt(n)$, 
$n=2,3,\dots$ with rational coefficients.
There is also a result going back to Euler \cite{E} that any
double zeta value $\zt(m,n)$ with $m+n$ odd can be so expressed; in fact
\begin{multline}
\label{eulerred}
\zt(m,n)=\frac12\left[(-1)^m\binom{m+n}{n}-1\right]\zt(m+n)
+\frac{1+(-1)^m}2\zt(m)\zt(n)\\
-(-1)^m\sum_{j=1}^{\frac{m+n-1}2}\left[\binom{2j-2}{m-1}+\binom{2j-2}{n-1}\right]
\zt(2j-1)\zt(m+n-2j+1)
\end{multline}
for $m+n$ odd and $n>1$, where it is understood that $\binom{p}{q}$
is nil if $p<q$.  (If $n=1$ then equation (\ref{n1red}) applies.)
But in even weights 8 and higher there appear to be double zeta values
that cannot be expressed as rational polynomials in the ordinary zeta 
values, e.g., $\zt(2,6)$.
This phenomenon shows up in \cite[Table 9]{Fr}, where a ``new constant'' 
$\kappa_{16}$ (which is $-\zt(2,6)$ by equation (\ref{fr}) below) 
must be introduced to give formulas for weight-8 quantities.
\par
While most of the study of multiple zeta values has concentrated
on the case where the arguments are positive integers, various authors 
have studied the behaviour of $\zt(i_1,\dots,i_k)$ when the $i_k$ are 
allowed to be complex.  For a survey see \cite{M};
in particular, the series (\ref{mzv}) converges absolutely provided
$\Re(i_1+\dots+i_j)>j$ for $j=1,\dots,k$.
\section{Double products}
Our result for products of two tails is as follows.
\begin{thm}
\label{sym}
If $p,q>1$ and $p+q>3$, then
\begin{multline*}
\sum_{n=1}^\infty\left(\zt(p)-\sum_{j=1}^n\frac1{j^p}\right)
\left(\zt(q)-\sum_{k=1}^n\frac1{k^q}\right)=\\
\zt(p,q-1)+\zt(q,p-1)+\zt(p+q-1)-\zt(p)\zt(q) .
\end{multline*}
\end{thm}
\begin{proof}
Rearrange
\[
\sum_{n=1}^\infty\left(\zt(p)-\sum_{j=1}^n\frac1{j^p}\right)
\left(\zt(q)-\sum_{k=1}^n\frac1{k^q}\right)=
\sum_{n=1}^\infty\sum_{j=n+1}^\infty\sum_{k=n+1}^\infty\frac1{j^pk^q}
\]
as 
\begin{multline*}
\sum_{j,k\ge 2}\frac{\min\{j,k\}-1}{j^pk^q}=
\sum_{k=2}^\infty\frac{k-1}{k^{p+q}}
+\sum_{j=2}^\infty\frac{j-1}{j^p}\sum_{k=j+1}^\infty\frac1{k^q}
+\sum_{k=2}^\infty\frac{k-1}{k^q}\sum_{j=k+1}^\infty\frac1{j^p}\\
=\sum_{k=1}^\infty\frac{k-1}{k^{p+q}}
+\sum_{j=1}^\infty\sum_{k>j}\frac{j-1}{j^pk^q}
+\sum_{k=1}^\infty\sum_{j>k}\frac{k-1}{j^pk^q}\\
=\zt(p+q-1)-\zt(p+q)+\zt(q,p-1)-\zt(q,p)+\zt(p,q-1)-\zt(p,q)\\
=\zt(p+q-1)+\zt(q,p-1)+\zt(p,q-1)-\zt(p)\zt(q) .
\end{multline*}
\end{proof}
Setting $q=p$ in this result gives the following.
\begin{cor}
\label{psq}
If $p>\frac32$, then 
\[
\sum_{n=1}^\infty\left(\zt(p)-\sum_{j=1}^n\frac1{j^p}\right)^2=
2\zt(p,p-1)+\zt(2p-1)-\zt(p)^2 .
\]
\end{cor}
If $p\ge 3$ is an integer, we can combine the corollary with equation
(\ref{eulerred}) to get
\begin{multline}
\label{pp3}
\sum_{n=1}^\infty\left(\zt(p)-\sum_{j=1}^n\frac1{j^p}\right)^2=
(-1)^p\binom{2p-1}{p}\zt(2p-1)+((-1)^p+1)\zt(p)\zt(p-1)\\
-2(-1)^p\sum_{j=1}^{p-1}\binom{2j-1}{p-1}\zt(2j-1)\zt(2p-2j)-\zt(p)^2 .
\end{multline}
This is equivalent to the formula given by \cite[Theorem 1(b)]{FS}.
\par
Theorem 1 of \cite{FV} states that for integer $k\ge 3$,
\begin{multline}
\label{kk1}
\sum_{n=1}^\infty\left(\zt(k)-\sum_{i=1}^n\frac1{i^k}\right)
\left(\zt(k+1)-\sum_{j=1}^n\frac1{j^{k+1}}\right)=\\
\frac12\zt(k)^2+\frac12\zt(2k)-\zt(k)\zt(k+1)+
\frac{(-1)^k}{k!}\int_0^1\frac{\log^kx\Li_{k-1}(x)}{1-x}dx ,
\end{multline}
where $\Li_p$ is the polylogarithm
\[
\Li_p(z)=\sum_{j=1}^\infty \frac{z^j}{j^p} .
\]
(In fact this formula holds for $k=2$, as we see below.)
Equation (\ref{kk1}) is based on \cite[Lemma 2.1]{Fr}, which implies
\[
\zt(q,r)+\zt(q+r)=\zt(q)\zt(r)-\frac{(-1)^{r-1}}{(r-1)!}\int_0^1
\frac{\log^{r-1}(x)\Li_q(x)}{1-x}dx
\]
for integers $q,r\ge 2$, or
\begin{equation}
\label{fr}
\zt(r,q)=\frac{(-1)^{r-1}}{(r-1)!}\int_0^1 \frac{\log^{r-1}(x)\Li_q(x)}{1-x}dx
\end{equation}
using the product relation (\ref{prod}).  Indeed equation (\ref{fr})
holds for integers $r\ge 2$, $q\ge 1$, as can be shown directly in 
the case $q=1$ using $\Li_1(x)=-\log(1-x)$.  But by writing equation
(\ref{fr}) in the form
\begin{equation}
\label{hf}
\zt(r,q)=\frac1{\Ga(r)}\int_0^\infty \frac{t^{r-1}\Li_q(e^{-t})}{e^t-1}dt 
\end{equation}
we can do better; we can eliminate the requirement that $r$ and $q$ be
integers.
\begin{lem}
Equation (\ref{hf}) holds for real $r>1$ and $q>2-r$.
\end{lem}
\begin{proof}
We have
\begin{multline*}
\frac{t^{r-1}\Li_q(e^{-t})}{e^t-1}
=\frac{e^{-t}t^{r-1}}{1-e^{-t}}\sum_{k=1}^\infty\frac{e^{-kt}}{k^q}
=\sum_{j=1}^\infty e^{-jt}\sum_{k=1}^\infty\frac{t^{r-1}e^{-kt}}{k^q}\\
=\sum_{k=1}^\infty\frac1{k^q}\sum_{j=1}^\infty t^{r-1}e^{-(j+k)t} ,
\end{multline*}
and if $r>1$, $q>2-r$ we can integrate to get the convergent series
\[
\int_0^\infty \frac{t^{r-1}\Li_q(e^{-t})}{e^t-1}dt=\sum_{k=1}^\infty\frac1{k^q}
\sum_{j=1}^\infty \int_0^\infty t^{r-1}e^{-(j+k)t}dt=
\sum_{k,j\ge 1}\frac{\Ga(r)}{k^q(j+k)^r} ,
\]
from which the result follows.
\end{proof}
\par\noindent
{\bf Remark.} It is interesting to compare equation (\ref{hf}) with 
the Arakawa-Kaneko zeta function \cite{AK}
\[
\xi_k(s)=\frac1{\Ga(s)}\int_0^\infty\frac{t^{s-1}\Li_k(1-e^{-t})}{e^t-1}dt ,
\]
which in the case where $k,s$ are positive integers is known \cite{O}
to be equal to the $s$-fold sum
\[
\sum_{n_1\ge n_2\ge\cdots\ge n_s\ge 1}\frac1{n_1^{k+1}n_2\cdots n_s} .
\]
\par
We now use the lemma to give a generalized version of formula 
(\ref{kk1}).
\begin{prop}
For real $k>1$,
\begin{multline*}
\sum_{n=1}^\infty\left(\zt(k)-\sum_{i=1}^n\frac1{i^k}\right)
\left(\zt(k+1)-\sum_{j=1}^n\frac1{j^{k+1}}\right)=\\
\frac12\zt(k)^2+\frac12\zt(2k)-\zt(k)\zt(k+1)+
\frac1{\Ga(k+1)}\int_0^\infty\frac{t^k\Li_{k-1}(e^{-t})}{e^t-1}dt .
\end{multline*}
\end{prop}
\begin{proof} Equation (\ref{hf}) implies
\[
\zt(k+1,k-1)=\frac1{\Ga(k+1)}\int_0^\infty \frac{t^k\Li_{k-1}(e^{-t})}{e^t-1}dt
\]
for $k>1$.  Since Theorem \ref{sym} above gives
\begin{multline*}
\sum_{n=1}^\infty\left(\zt(k)-\sum_{i=1}^n\frac1{i^k}\right)
\left(\zt(k+1)-\sum_{j=1}^n\frac1{j^{k+1}}\right)=\\
\zt(k,k)+\zt(k+1,k-1)+\zt(2k)-\zt(k)\zt(k+1) 
\end{multline*}
for $k>1$, the conclusion follows since 
$\zt(k,k)+\zt(2k)=\frac12[\zt(k)^2+\zt(2k)]$
(a case of equation (\ref{prod})).
\end{proof}
There is a similar result that comes from setting $p=k$ in 
Corollary \ref{psq} and using equation (\ref{hf}).
\begin{prop}
For real $k>\frac32$, 
\[
\sum_{n=1}^\infty\left(\zt(k)-\sum_{i=1}^n\frac1{i^k}\right)^2=
\zt(2k-1)-\zt(k)^2+\frac2{\Ga(k)}\int_0^\infty \frac{t^{k-1}\Li_{k-1}(e^{-t})}{e^t-1}
dt .
\]
\end{prop}
\par
We now examine some special cases.
If $p=2$, Corollary \ref{psq} gives
\begin{equation}
\label{2sp2}
\sum_{n=1}^\infty\left(\zt(2)-\sum_{j=1}^n\frac1{j^2}\right)^2=
2\zt(2,1)+\zt(3)-\zt(2)^2=3\zt(3)-\frac52\zt(4),
\end{equation}
using the identity $\zt(2,1)=\zt(3)$ implied
by equations (\ref{2sumt}) or (\ref{n1red}).
Cf. \cite[Prob. 3.22]{F} and \cite[Theorem 1(a)]{FS}.
For $(p,q)=(3,2)$ Theorem \ref{sym} gives 
\begin{multline}
\label{e23}
\sum_{n=1}^\infty\left(\zt(3)-\sum_{j=1}^n\frac1{j^3}\right)
\left(\zt(2)-\sum_{k=1}^n\frac1{k^2}\right)=
\zt(3,1)+\zt(2,2)+\zt(4)-\zt(2)\zt(3)\\
=2\zt(4)-\zt(2)\zt(3),
\end{multline}
where we have used equation (\ref{2sumt}) with $n=4$.
This agrees with \cite[Theorem 1]{FV}, and also with equation 
(\ref{kk1}) as extended to the case $k=2$.
For $p=3$ equation (\ref{pp3}) is
\[
\sum_{n=1}^\infty\left(\zt(3)-\sum_{j=1}^n\frac1{j^3}\right)^2=
-10\zt(5)+6\zt(3)\zt(2)-\zt(3)^2;
\]
cf. \cite[Cor. 2]{FS}.
In the case $(p,q)=(4,3)$ Theorem \ref{sym} gives (cf. \cite[Cor. 2]{FV})
\begin{multline*}
\sum_{n=1}^\infty\left(\zt(4)-\sum_{j=1}^n\frac1{j^4}\right)
\left(\zt(3)-\sum_{k=1}^n\frac1{k^3}\right)=
\zt(4,2)+\zt(3,3)+\zt(6)-\zt(3)\zt(4)\\
=-\frac56\zt(6)+\frac32\zt(3)^2-\zt(3)\zt(4) .
\end{multline*}
\section{Triple and higher products}
We now consider the sum of triple products
\begin{multline}
\label{3sum}
\sum_{n=1}^\infty \left(\zt(p)-\sum_{i=1}^n\frac1{i^p}\right)
\left(\zt(q)-\sum_{j=1}^n \frac1{j^q}\right) 
\left(\zt(r)-\sum_{k=1}^n \frac1{k^r}\right) =\\
\sum_{n=1}^\infty\sum_{i=n+1}^\infty\sum_{j=n+1}^\infty\sum_{k=n+1}^\infty\frac1{i^pj^qk^r} .
\end{multline}
The result is as follows.
\begin{thm}
\label{triple}
If $p,q,r>1$ and $p+q+r>4$, then the sum (\ref{3sum}) is
\begin{multline*}
\zt(p,q,r-1)+\zt(p,r,q-1)+\zt(q,p,r-1)+\zt(q,r,p-1)+\zt(r,p,q-1)\\
+\zt(r,q,p-1)+\zt(p+q,r-1)+\zt(p+r,q-1)+\zt(q+r,p-1)+\zt(p,q+r-1)\\
+\zt(q,p+r-1)+\zt(r,p+q-1)+\zt(p+q+r-1)-\zt(p)\zt(q)\zt(r) .
\end{multline*}
\end{thm}
\begin{proof}
Write the sum (\ref{3sum}) as
\begin{equation}
\label{sumijk}
\sum_{i,j,k\ge 1}\frac{\min\{i,j,k\}-1}{i^pj^qk^r}
\end{equation}
and consider all possible ``weak orderings'' of $i,j,k$ 
(i.e., ties are allowed, e.g., $i>j>k$ or $i=j>k$); there are thirteen 
possibilities in all.
These fall into four natural classes.  First, there are six
orderings with $i,j,k$ all unequal.  Second, there are three orderings
with two of $i,j,k$ equal and less than the other value; and third, 
there are three orderings with two of $i,j,k$ equal and greater than 
the other value.  Finally, there is one ordering with $i=j=k$.
For each ordering the corresponding terms in (\ref{sumijk}) 
add up to the difference of two multiple zeta values.
For example, $i>j>k$ contributes
\[
\sum_{i>j>k\ge 1}\frac{k-1}{i^pj^qk^r}=\zt(p,q,r-1)-\zt(p,q,r) ,
\]
and together with the other five orderings in the same class we have
\begin{multline}
\label{111}
\zt(p,q,r-1)+\zt(p,r,q-1)+\zt(q,p,r-1)+\zt(q,r,p-1)+\zt(r,p,q-1)\\
+\zt(r,q,p-1)
-\zt(p,q,r)-\zt(p,r,q)-\zt(q,p,r)-\zt(q,r,p)-\zt(r,p,q)-\zt(r,q,p) .
\end{multline}
Similarly, $i=j>k$ contributes 
\[
\sum_{j>k\ge 1}\frac{k-1}{j^{p+q}k^r}=\zt(p+q,r-1)-\zt(p+q,r) ,
\]
and together with the other two orderings in the same class we have
\begin{equation}
\label{21}
\zt(p+q,r-1)+\zt(p+r,q-1)+\zt(q+r,p-1)-\zt(p+q,r)-\zt(p+r,q)-\zt(q+r,p) .
\end{equation}
The ordering $i>j=k$ contributes
\[
\sum_{i>j\ge 1}\frac{j-1}{i^pj^{q+r}}=\zt(p,q+r-1)-\zt(p,q+r) ,
\]
which together with the other two orderings in the same class gives
\begin{equation}
\label{12}
\zt(p,q+r-1)+\zt(q,p+r-1)+\zt(r,p+q-1)-\zt(p,q+r)-\zt(q,p+r)-\zt(r,p+q).
\end{equation}
Finally, the ordering $i=j=k$ contributes
\begin{equation}
\label{3}
\sum_{i=1}^\infty \frac{i-1}{i^{p+q+r}}=\zt(p+q+r-1)-\zt(p+q+r) .
\end{equation}
When (\ref{111}), (\ref{21}), (\ref{12}), and (\ref{3}) are combined
the thirteen negative terms add up to $\zt(p)\zt(q)\zt(r)$, giving the
conclusion.
\end{proof}
In the case $p=q=r=2$, Theorem \ref{triple} gives
\begin{equation}
\label{2sp3}
\sum_{n=1}^\infty\left(\zt(2)-\sum_{i=1}^n\frac1{i^2}\right)^3=
6\zt(2,2,1)+3\zt(4,1)+3\zt(2,3)+\zt(5)-\zt(2)^3 .
\end{equation}
The right-hand side of this equation can be simplified using the
properties discussed in \S2.  First, by the duality property (\ref{dual}),
$\zt(2,2,1)=\zt(3,2)$.  Then we can write the sum as 
\begin{multline*}
3\zt(3,2)+3[\zt(4,1)+\zt(3,2)+\zt(2,3)]+\zt(5)-\zt(2)^3=
3\zt(3,2)+4\zt(5)-\zt(2)^3\\
=9\zt(2)\zt(3)-\frac{25}2\zt(5)-\frac{35}8\zt(6),
\end{multline*}
where we used equation (\ref{2sumt}) and then relations that follow
from (\ref{binom}).
Theorem \ref{triple} in the case $(p,q,r)=(3,2,2)$ is
\begin{multline*}
\sum_{n=1}^\infty\left(\zt(2)-\sum_{i=1}^n\frac1{i^2}\right)^2
\left(\zt(3)-\sum_{j=1}^n\frac1{j^3}\right)=
2\zt(3,2,1)+2\zt(2,3,1)\\
+2\zt(2,2,2)+2\zt(5,1)+2\zt(2,4)+\zt(3,3)+\zt(4,2)
+\zt(6)-\zt(2)^2\zt(3)\\
=\frac76\zt(6)+\frac32\zt(3)^2-\zt(2)^2\zt(3) ,
\end{multline*}
and in the case $(p,q,r)=(3,3,2)$ is
\begin{multline*}
\sum_{n=1}^\infty\left(\zt(2)-\sum_{i=1}^n\frac1{i^2}\right)
\left(\zt(3)-\sum_{j=1}^n\frac1{j^3}\right)^2=
2\zt(3,3,1)+2\zt(3,2,2)\\
+2\zt(2,3,2)+\zt(6,1)+2\zt(5,2)+2\zt(3,4)+\zt(2,5)
+\zt(7)-\zt(2)\zt(3)^2\\
=
\frac{77}8\zt(7)+3\zt(2)^2\zt(3)-10\zt(2)\zt(5)-\zt(2)\zt(3)^2 .
\end{multline*}
\par
Theorems \ref{sym} and \ref{triple} are special cases of a result 
for $k$-fold products, but the number of terms in the general case is one 
plus the number of weak orderings on $k$ objects \cite[A000670]{S},
which increases rapidly with $k$.
We can state it concisely by introducing some more notation.  Let $\Si_k$ be
the symmetric group on $k$ letters, $\C_k$ the set of compositions of $k$.
For a sequence $I=(i_1,\dots,i_k)$ of positive real numbers and an
element $J=(j_1,\dots,j_p)\in\C_k$, let 
\[
\zt(I,J)=\frac{\zt(J_1(I),J_2(I),\dots,J_{p-1}(I),J_p(I))}{j_1!\cdots j_p!}
\]
and
\[
\hat\zt(I,J)=\frac{\zt(J_1(I),J_2(I),\dots,J_{p-1}(I),J_p(I)-1)}
{j_1!\cdots j_p!},
\]
where $J_c(I)=i_{j_1+\dots+j_{c-1}+1}+\dots+i_{j_1+\dots+j_c}$ for $c=1,\dots,p$.
The following lemma generalizes the product formula (\ref{prod}).
\begin{lem}
\label{mprod}
For $i_1,\dots,i_k>1$,
\[
\prod_{j=1}^k\zt(i_j)=\sum_{\si\in\Si_k}\sum_{J\in\C_k}
\zt((i_{\si(1)},\dots,i_{\si(k)}),J) .
\]
\end{lem}
\begin{proof}
The left-hand side is
\begin{equation}
\label{psum}
\sum_{m_1,\dots,m_k\ge 1}\frac1{m_1^{i_1}\cdots m_k^{i_k}} .
\end{equation}
Now consider all possible weak orderings on $\{m_1,\dots,m_k\}$.
As in the proof of Theorem \ref{triple}, these fall into natural
classes which correspond to compositions of $k$.  
Given $(j_1,\dots,j_p)\in\C_k$, any ordering in the corresponding
class can be obtained by applying a permutation of $\{1,\dots,k\}$ to 
\[
m_1=\dots=m_{j_1}<m_{j_1+1}=\dots=m_{j_1+j_2}<\dots<m_{j_1+\dots+j_{p-1}+1}=\dots=m_k .
\]
But swaps of equal elements don't matter, so the number of 
distinct weak orderings in the class of $(j_1,\dots,j_p)$ is
\[
\frac{k!}{j_1!\cdots j_p!} .
\]
Evidently this class contributes 
\[
\sum_{\si\in\Si_k}\zt((i_{\si(1)},\dots,i_{\si(k)}),(j_1,\dots,j_p))
\]
to the sum (\ref{psum}), and the result follows by summing over $\C_k$.
\end{proof}
Our general result for $k$-fold products is as follows.
\begin{thm}
\label{kfold}
For $i_1,i_2,\dots,i_k>1$ such that $i_1+\dots+i_k>k+1$,
\[
\sum_{n=1}^\infty\prod_{j=1}^k\left(\zt(i_j)-\sum_{m=1}^n\frac1{m^{i_j}}\right)=
\sum_{\si\in\Si_k}\sum_{J\in\C_k}\hat\zt((i_{\si(1)},\dots,i_{\si(k)}),J)
-\prod_{j=1}^k\zt(i_j) .
\]
\end{thm}
\begin{proof}
The proof is similar in outline to that for Theorem \ref{triple}.  
First write the left-hand side as
\begin{equation}
\label{msum}
\sum_{m_1,\dots,m_k\ge 1}\frac{\min\{m_1,\dots,m_k\}-1}{m_1^{i_1}\cdots m_k^{i_k}}
\end{equation}
and separate out the terms according to all the possible
weak orderings of $m_1,\dots,m_k$.
These orderings fall into classes specified by compositions of $k$.
Given such a composition $(j_1,\dots,j_p)$, the contribution from the 
corresponding class of weak orderings to the sum (\ref{msum}) can be 
seen to be
\[
\sum_{\si\in\Si_k}\left[\hat\zt((i_{\si(1)},\dots,i_{\si(k)}),(j_1,\dots,j_p)) - 
\zt((i_{\si(1)},\dots,i_{\si(k)}),(j_1,\dots,j_p))\right] ,
\]
and the conclusion follows using Lemma \ref{mprod}.
\end{proof}
\begin{cor}
\label{rep}
For integer $k\ge 2$ and real $r>1+\frac1k$,
\[
\sum_{n=1}^\infty\left(\zt(r)-\sum_{i=1}^n\frac1{i^r}\right)^k=
\sum_{J=(j_1,\dots,j_p)\in\C_k}\binom{k}{J}\zt(rj_1,\dots,rj_{p-1},rj_p-1)-\zt(r)^k ,
\]
where $\binom{k}{J}$ is the multinomial coefficient.
\end{cor}
For $r=2$, the first two cases of Corollary \ref{rep} are equations 
(\ref{2sp2}) and (\ref{2sp3}); the next is
\begin{multline*}
\sum_{n=1}^\infty\left(\zt(2)-\sum_{i=1}^n\frac1{i^2}\right)^4=
24\zt(2,2,2,1)+12\zt(4,2,1)+12\zt(2,4,1)\\
+12\zt(2,2,3)+4\zt(6,1)+4\zt(2,5)+6\zt(4,3)+\zt(7)-\zt(2)^4 \\
=-\frac{301}4\zt(7)+10\zt(2)\zt(5)+\frac{102}5\zt(3)\zt(2)^2
-\frac{175}{24}\zt(8).
\end{multline*}
\section*{Acknowledgement}
The author thanks the referee for his careful reading of the manuscript,
and for bringing several references to his attention.

\end{document}